\documentclass[letterpaper, 10 pt, conference]{ieeeconf}  

\IEEEoverridecommandlockouts  

\usepackage{amsthm} 
\newtheorem{thm}{Theorem}
\newtheorem{corr}{Corollary}

\newtheorem{rem}{Remark}
\newtheorem{assum}{Assumption}
\newtheorem{prop}{Proposition}
\newtheorem{defn}{Definition}
\newtheorem{problem}{Problem}
\newtheorem{subproblem}{Problem}

\usepackage{caption}
\usepackage{graphicx}
\usepackage{subcaption}
\graphicspath{{../Figures/}}

\usepackage{subcaption}
\usepackage{hyperref}
\usepackage{cite}
\usepackage{multirow}
\usepackage[export]{adjustbox}
\usepackage{amssymb} 
\usepackage{amsbsy}
\usepackage{amsmath} 
\usepackage{bm}
\usepackage{mathtools}
\usepackage{cases}
\usepackage{nccmath}
\usepackage{mathrsfs}
\usepackage{upgreek}
\usepackage{adjustbox}
\usepackage{tikz}
\usepackage{bigdelim}
\usepackage{comment}
\usepackage{lipsum}
\usepackage{subcaption}

\newenvironment{problem*}[1]
{%
    \ifcsname problem@#1\endcsname
    \expandafter\xdef\csname problem@#1\endcsname
    {%
        \the\numexpr\csname problem@#1\endcsname+1\relax
    }%
    \else
    \expandafter\xdef\csname problem@#1\endcsname{1}%
    \fi
    \edef\thesubproblem
    {%
        \getrefnumber{#1}.\csname problem@#1\endcsname
    }%
    \subproblem
}
{
    \endsubproblem
}

\usetikzlibrary{shapes,arrows, intersections, fit}

\usepackage{algorithm}
\usepackage{algpseudocode}

\algdef{SE}[DOWHILE]{Do}{doWhile}{\algorithmicdo}[1]{\algorithmicwhile\ #1}%

\makeatletter
\define@key{Gin}{Trim}
            {\let\Gin@viewport@code\Gin@trim\expandafter\Gread@parse@vp#1 \\}
\makeatother

\renewcommand{\d}{\mathrm{d}}
\newcommand{\E}{\mathbb{E}}
\renewcommand{\P}{\mathbb{P}}
\newcommand{\R}{\mathbb{R}}
\newcommand{\N}{\mathbb{N}}

\newcommand{\bfy}{\mathbf{y}}
\newcommand{\bfz}{\mathbf{z}}
\newcommand{\bfx}{\mathbf{x}}
\newcommand{\bfX}{\mathbf{X}}
\newcommand{\bfu}{\mathbf{u}}
\newcommand{\bfU}{\mathbf{U}}
\newcommand{\bfw}{\mathbf{w}}
\newcommand{\bfW}{\mathbf{W}}
\newcommand{\bfc}{\mathbf{c}}
\newcommand{\bfd}{\mathbf{d}}

\newcommand{\scriptA}{\mathcal{A}}
\newcommand{\scriptB}{\mathcal{B}}
\newcommand{\scriptQ}{\mathcal{Q}}
\newcommand{\scriptR}{\mathcal{R}}
\newcommand{\scriptD}{\mathcal{D}}

\renewcommand{\i}{\mathrm{i}}
\renewcommand{\Re}{\mathbb{R}}

\begin{document}
\title{Distribution Steering for Discrete-Time Linear Systems with General Disturbances using Characteristic Functions}

\author{Vignesh Sivaramakrishnan, Joshua Pilipovsky, Meeko Oishi, Panagiotis Tsiotras
\thanks{%
V. Sivaramakrishnan is a graduate student with Electrical and Computer Engineering, University of New Mexico, Albuquerque, NM. Email: vigsiv@unm.edu\newline\indent
J. Pilipovsky is a graduate student at the School of Aerospace Engineering, Georgia Institute of Technology, Atlanta, GA. Email: jpilipovsky3@gatech.edu\newline\indent
M. Oishi is a Professor with Electrical and Computer Engineering, University of New Mexico, Albuquerque, NM. Email: oishi@unm.edu\newline\indent
P. Tsiotras is the Andrew and Lewis Chair Professor with the D. Guggen-heim School of Aerospace Engineering, and the Institute for Robotics and Intelligent Machines, Georgia Institute of Technology, Atlanta, GA. Email: tsiotras@gatech.edu
}}
\maketitle

\begin{abstract}
We propose to solve a constrained distribution steering problem, i.e., steering a stochastic linear system from an initial distribution to some final, desired distribution subject to chance constraints.
We do so by characterizing the cumulative distribution function in the chance constraints and by using the absolute distance between two probability density functions using the corresponding characteristic functions.
We consider discrete-time, time-varying linear systems with affine feedback. 
We demonstrate the proposed approach on a 2D double-integrator perturbed by various disturbances and initial conditions. 
\end{abstract} 

\section{Introduction}

In many autonomous systems, the uncertainties that affect the  system evolution are quite complex and non-Gaussian.
For example, in urban air mobility scenarios, uncertainties may arise from the sensing and perception subsystems, the operating environment (i.e., wind gusts, ground effects), the presence of humans in the loop, or from unmodeled physical phenomena (i.e., higher-order nonlinearities in lift or drag forces) among many others.  
Methods that ensure robustness to such uncertainties are important for improving the reliability and robustness of autonomous systems. 
Recently, distribution steering, by which a controller manipulates the stochasticity of the state to guide at states towards a desirable probability distribution, has emerged as a promising approach to \textit{directly} control system uncertainty~\cite{Halder2,balci2020covariance,Chen1,Williams2,Williams1}. 
However, most of these works to date assume Gaussian noise,
In this paper, we describe an approach to synthesize feedback controllers for distribution steering of discrete-time linear systems subject to general (e.g., non-Gaussian) disturbances.  

Recent work in covariance steering, in which a system is steered from an initial Gaussian distribution to a desired Gaussian distribution, captures the covariance and the mean as extended state variables \cite{Max1,pp_kazu,bakolas2018finite,Chen1,JP1,JoshJack}.
However, extending this approach to non-Gaussian disturbances is not straightforward.
For non-Gaussian disturbances, the distribution is characterized by higher order moments and the
computational complexity of the problem increases as the number of moments we need to steer increases. 
Additionally, when incorporating chance, or probabilistic, constraints into the problem formulation, it is not clear how one can make these constraints tractable (for example, as second order cone constraints through Boole's inequality \cite{PP_blackmore,Boole}) with non-Gaussian state evolution, since closed form expressions often do not exist. 
Methods in optimal transport theory \cite{Halder2,Chen1} can steer to and from arbitrary distributions, but presume that the disturbance is a Wiener process (i.e., Gaussian increments).  

\begin{figure}
    \centering
    \pgfmathsetseed{5}
\begin{tikzpicture}
    \foreach \fscale in {1,0.8,...,0.2} {
        \begin{scope}[scale=\fscale]
            \draw [line width=1.6pt-\fscale, black, tension=0.5] plot [smooth cycle] coordinates {
                (-1, 0)
                (-0.9, 0.8)
                (-0.2, 0.6)
                (0.2, 1)
                (0.5, 0.9)
                (1, 0.5)
                (0.8, 0)
                (1, -1)
                (0.2, -1)
                (-0.6, -0.2)
            };
        \end{scope}
    }

    \draw [fill=cyan, fill opacity=0.2, tension=0.5] plot [smooth cycle] coordinates {
        (-1, 0)
        (-0.9, 0.8)
        (-0.2, 0.6)
        (0.2, 1)
        (0.5, 0.9)
        (1, 0.5)
        (0.8, 0)
        (1, -1)
        (0.2, -1)
        (-0.6, -0.2)
    };

    \begin{scope}[xshift=6cm]
        \foreach \fscale in {1,0.8,...,0.2} {
            \begin{scope}[scale=\fscale]
                \draw [line width=1.6pt-\fscale, black, tension=0.8] plot [smooth cycle] coordinates {
                    (-0.8, 0.8)
                    (0.5, 1)
                    (0.8, 0.7)
                    (0.7, 0.2)
                    (0.9, -0.5)
                    (0, -1)
                    (-0.2, -0.8)
                    (-1, -1)
                };
            \end{scope}
        }

        \draw [fill=yellow, fill opacity=0.2, tension=0.8] plot [smooth cycle] coordinates {
            (-0.8, 0.8)
            (0.5, 1)
            (0.8, 0.7)
            (0.7, 0.2)
            (0.9, -0.5)
            (0, -1)
            (-0.2, -0.8)
            (-1, -1)
        };
    \end{scope}

    \draw [line width=1.2pt, orange] (3.5, 1) -- (3.2, 1.1) -- (3.75, -2) -- (4.05, -2.1) -- (3.5, 1) -- (3.2, 1.1) -- (3.75, -2) -- (4.05, -2.1) -- cycle;

    \draw [line width=1.2pt, orange] (2.5, 0.6) -- (2.25, 0.8) -- (1.75, -0.8) -- (2, -1) -- (2.5, 0.6) -- (2.25, 0.8) -- (1.75, -0.8) -- (2, -1) -- cycle;

    \draw [line width=1.2pt, orange] (3.5, 1) -- (4.05, -2.1);

    \draw [line width=1.2pt, orange] (2.5, 0.6) -- (2, -1);

    \node[circle, magenta, fill, inner sep=1.5pt] at (-0.8, 0.4) {};
    \node[circle, magenta, fill, inner sep=1.5pt] at (0.25, 0.5) {};
    \node[circle, magenta, fill, inner sep=1.5pt] at (0.5, -0.5) {};

    \draw [line width=1.5pt, magenta, -latex, tension=0.5] plot [smooth] coordinates{
        (-0.8, 0.4)
        (0 + 0.1*rand, 1 + 0.1*rand)
        (1 + 0.1*rand, 0.8 + 0.1*rand)
        (1.5 + 0.1*rand, 0.5 + 0.1*rand)
        (2.5 + 0.1*rand, 0.5 + 0.1*rand)
        (3 + 0.1*rand, 0.5 + 0.1*rand)
        (3.5 + 0.1*rand, 0.5 + 0.1*rand)
        (4 + 0.1*rand, 0.5 + 0.1*rand)
        (4.5 + 0.1*rand, 0.8 + 0.1*rand)
        (5 + 0.1*rand, 0.8 + 0.1*rand)
        (5.5 + 0.1*rand, 0.8 + 0.1*rand)
        (6 + 0.1*rand, 0.5 + 0.1*rand)
    };
    \draw [line width=1.5pt, magenta, -latex, tension=0.5] plot [smooth] coordinates{
        (0.25, 0.5)
        (0.5 + 0.1*rand, 0.5 + 0.1*rand)
        (1 + 0.1*rand, 0 + 0.1*rand)
        (1.5 + 0.1*rand, 0 + 0.1*rand)
        (2 + 0.1*rand, 0 + 0.1*rand)
        (2.5 + 0.1*rand, 0 + 0.1*rand)
        (3 + 0.1*rand, 0 + 0.1*rand)
        (3.5 + 0.1*rand, 0 + 0.1*rand)
        (4 + 0.1*rand, 0 + 0.1*rand)
        (4.5 + 0.1*rand, 0 + 0.1*rand)
        (5 + 0.1*rand, 0.5 + 0.1*rand)
        (5.5 + 0.1*rand, 0.5 + 0.1*rand)
        (6 + 0.1*rand, 0.1 + 0.1*rand)
    };
    \draw [line width=1.5pt, magenta, -latex, tension=0.5] plot [smooth] coordinates{
        (0.5, -0.5)
        (1 + 0.1*rand, -0.5 + 0.1*rand)
        (1.5 + 0.1*rand, -0.5 + 0.1*rand)
        (2 + 0.1*rand, -0.5 + 0.1*rand)
        (2.5 + 0.1*rand, -0.5 + 0.1*rand)
        (3 + 0.1*rand, -1 + 0.1*rand)
        (3.5 + 0.1*rand, -0.8 + 0.1*rand)
        (4 + 0.1*rand, -1.25 + 0.1*rand)
        (4.5 + 0.1*rand, -1 + 0.1*rand)
        (5 + 0.1*rand, -0.5 + 0.1*rand)
        (5.5 + 0.1*rand, -0.5 + 0.1*rand)
    };

    \draw [line width=1pt, -stealth, shorten >=0.1cm, shorten <=0.1cm] (1, -1) to [bend right=40pt] node [midway, below, xshift=-1cm, yshift=-0.1cm] {\scriptsize $\bfX = \scriptA \bfx_0 + \scriptB \bfU + \mathcal{D} \bfW$} (5, -1);

    \draw [line width=2.2pt, white] (3.2, 1.1) -- (3.75, -2);

    \draw [line width=1.2pt, orange] (3.2, 1.1) -- (3.75, -2);
    \draw [line width=1.2pt, orange] (3.5, 1) -- (3.2, 1.1);
    \draw [line width=1.2pt, orange] (3.75, -2) -- (4.05, -2.1);

    \draw [line width=2.2pt, white] (2.25, 0.8) -- (1.75, -0.8);

    \draw [line width=1.2pt, orange] (2.25, 0.8) -- (1.75, -0.8);
    \draw [line width=1.2pt, orange] (2.5, 0.6) -- (2.25, 0.8);
    \draw [line width=1.2pt, orange] (1.75, -0.8) -- (2, -1);

    \node [draw=none, orange] at (4.3, -1.85) {\scriptsize $\mathcal{X}_{k}$};

    \node [draw=none] at (0.5, -1.3) {\scriptsize $\psi_{\mathbf{x}_{0}}$};
    \node [draw=none] at (6, -1.3) {\scriptsize $\psi_{\mathbf{x}_{f}}$};

\end{tikzpicture}
    \caption{We seek to steer a stochastic system from an initial distribution to a desired final distribution, subject to probabilistic constraints on the state and input.}
    \label{fig:dist_steering_visual}
\end{figure}
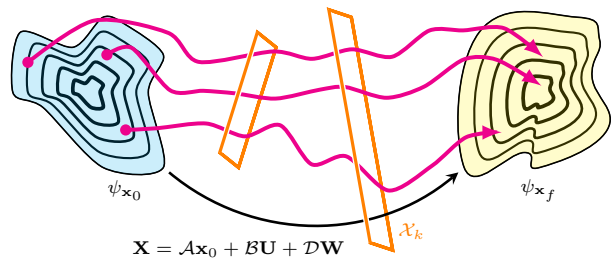

In this paper, we formulate the problem of steering non-Gaussian distributions under affine feedback as the minimization of the absolute distance between the final and desired state probability densities.  We consider a broad class of disturbances for linear systems, which places very few assumptions on the disturbance pdf.
Our approach employs characteristic functions, which
circumvent the need to steer all moments individually, and enable straightforward calculation of the absolute distance between distributions.
The key insight is that expressions for the chance constraints (previously implemented in an open-loop context~\cite{sivaramakrishnan2020convexified,ECF}) and the terminal absolute distance constraint can be represented efficiently using characteristic functions and their compositions.

Section~\ref{sec:prelimform} reviews some preliminaries and formulates the problem we wish to solve. 
We present our analysis of chance constraints within the framework of characteristic functions in Section~\ref{sec:chanceconstraints}.
The terminal density matching constraint is presented in \ref{sec:terminal_density_constraints}.
Section \ref{sec:resulting_optimization} presents our reformulation of the constrained stochastic optimal control problem using characteristic functions. 
Section~\ref{sec:examples} presents an example of a 2D double integrator under various disturbances and initial conditions. 

\section{Preliminaries and Problem formulation}\label{sec:prelimform}

\subsection{Notation and Definitions}

We denote real-valued vectors with lowercase $u \in \mathbb R^m$, matrices with uppercase $V \in \mathbb R^{n \times m}$, and random vectors with bold case $\mathbf{w}\in\R^{p}$.
The $n$-dimension identity matrix is denoted by $I_{n}$, and the $m\times n$ dimensional zero matrix is denoted by $0_{m,n}$.
We define a diagonal matrix as $V = \mathrm{diag}(u)$ and a block diagonal matrix as $V = \mathrm{diag}(V_1,\cdots,V_i\cdots,V_n)$. 
The imaginary unit is denoted by $\mathrm{i}$; given a complex vector $\upvarphi\in\mathbb{C}^p$, its conjugate is denoted by $\overline\upvarphi$.
We denote intervals of integers using $\mathbb{N}_{[a,b]}$ where $a,b\in\mathbb{N},\; a<b$.
The vector $e_{i,d} = [0 \ \cdots \ 1 \ \cdots \ 0]^\intercal\in\R^{d}$ is a basis vector for $\R^{d}$ and isolates the $i$th component of a vector $\psi\in\R^{d}$ by $\psi_{i} = e_{i,d}^\intercal \psi$.

For a random vector $\mathbf{w}$, the probability space is $(\Omega,\mathscr{B}(\Omega),\mathbb{P}_{\mathbf{w}})$ with  $\Omega$ the set of all possible outcomes, $\mathscr{B}(\Omega)$ the Borel $\sigma$-algebra on $\Omega$, and $\mathbb{P}_{\mathbf{w}}$ the probability measure on $\mathscr{B}(\Omega)$ \cite[Sec. 2]{billingsley_probability_2012}.
We consider only random vectors that are continuous, i.e., with probability measure $\mathbb{P}\left(\{\bfw \in \mathcal{S}\}\right) = \int_{\mathcal{S}} \psi_{\bfw}(z)\ \d z$
for $\mathcal{S}\in\mathscr{B}(\Omega)$, and 
probability density function (pdf) $\psi_{\mathbf{w}}$ that satisfies
$\psi_{\mathbf{w}}\geq 0$ almost everywhere (a.e.) and $\int_{\mathbb{R}} \psi_{\mathbf{w}}(z)\ \d z = 1$.
For a random variable $\mathbf{y} = a^\intercal\bfw,\ a\in\mathbb{R}^p$, we denote 
$\mathbb{P}\{a^\intercal\bfw \leq \alpha \}$ by the cumulative distribution function (cdf) $\Phi_{a^\intercal\bfw}:\mathbb{R}\rightarrow[0,1]$ via $\mathbb{P}\{a^\intercal\bfw \leq \alpha \} =  \Phi_{a^\intercal\bfw}(\alpha)$,
which follows by definition \cite[Sec. 14]{billingsley_probability_2012}. 
We write $\bfw \sim \psi_{\mathbf{w}}$ to denote the fact that $\bfw$ is distributed according to the pdf $\psi_{\mathbf{w}}$.
We define the Lebesgue space of measurable pdfs with bounded $d$-norm by $L^d(\mathbb{R}^n)$ where $1\leq d < \infty$.
The Lebesgue norm of a probability density function $\psi_{\bfw}$ is $\|\psi_{\bfw}\|_d = \left(\int_{\mathbb{R}^p} \left|\psi_{\bfw}(z)\right|^d \d z\right)^{1/d}$.
The space of all (continuous) probability density functions forms a subset of $L_1(\R^p)$ since $\psi_{\bfw} \geq 0$ a.e. and $\int_{\R^p} \psi_{\bfw} \ \d z = 1$.

\subsection{Problem Formulation}

Consider the discrete, linear time-varying system 
\begin{equation}
    \mathbf{x}_{k+1} = A_k\mathbf{x}_k + B_k\mathbf{u}_k + D_k\mathbf{w}_k, \quad k\in\N_{[0,N-1]},  \label{eq:discreteDynamics}
\end{equation}
with state $\mathbf{x}_k \in \mathcal{X}_k \subseteq \mathbb R^n$, control input $\mathbf{u}_k \in \mathcal{U}_k \subseteq \mathbb R^m$, disturbance $\mathbf{w}_k \sim \psi_{\bfw,k}$, and matrices $A_k,\ B_k,\ D_k$ of appropriate dimensions.
We assume that the system starts at $\mathbf{x}_0\sim\psi_{\mathbf{x}_0}$, and we seek to steer the final state at $N$ to a desired distribution $\bfx_{N}\sim\psi_{\bfx_f}$.

\begin{assum}
Random vectors and elements of random vectors are independent, but not necessarily identically distributed. 
\end{assum}

Given a deterministic reference trajectory $x_{d,k}\in\mathbb{R}^n$ for  all $k\in\N_{[0,N-1]}$,
we consider the quadratic cost 
\begin{align} \label{eq:COST}
J(\bfu_0,\ldots,\bfu_{N-1}) = 
    \sum_{k=0}^{N-1} \mathbb{E}[(\bfx_k-x_{d,k})^\intercal &Q_k (\bfx_k-x_{d,k}) \nonumber \\
    &+ \bfu_k^\intercal R_k \bfu_k],
\end{align}
where $Q_k\succeq 0$, and $R_k\succ 0$. 
Following the formulation in \cite{okamoto2018optimal}, 

we can concatenate the dynamics \eqref{eq:discreteDynamics} as
\begin{equation}\label{eq:concatdyn}
    \bfX = \scriptA \bfx_0 + \scriptB \bfU + \scriptD \bfW,
\end{equation}
where $\bfX = [\bfx_0^\intercal, \ldots, \bfx_N^{\intercal}]^\intercal\in\R^{(N+1)n}, \ \bfU = [\bfu_0^\intercal,\ldots,\bfu_{N-1}^\intercal]^\intercal\in\R^{Nm}, \ \bfW = [\bfw_0^\intercal,\ldots,\bfw_{N-1}^\intercal]^\intercal\in\R^{Np}$, 
for some  matrices $\scriptA\in\mathbb{R}^{(N+1)n\times n},\ \scriptB\in\mathbb{R}^{(N+1)n\times Nm}$, and $\scriptD\in\mathbb{R}^{(N+1)n\times Np}$. 
The concatenated disturbance follows the distribution $\bfW\sim\psi_{\bfW} = \prod^{N-1}_{k=0}\psi_{\bfw,k}$.
Probabilistic constraints are imposed on the state and input, namely,
\begin{subequations}
    \label{eq:sysCC}
    \begin{align}
        \P\left\{\bigwedge_{k=1}^{N}E_k\bfX\in\mathcal{X}_k\right\}\geq 1 - \Delta_X, \label{eq:stateCC} \\
        \P\left\{\bigwedge_{k=0}^{N-1}F_k\bfU\in\mathcal{U}_k\right\}\geq 1 - \Delta_U, \label{eq:inputCC}
    \end{align}
\end{subequations}%
where $\mathcal{X}_k = \cap_{j=1}^{N_X} \{x\in\R^{n} : \alpha_{j,k}^{\intercal} x \leq \beta_{j,k}\}$ and $\mathcal{U}_k = \cap_{j=1}^{N_U}\{u : a_{j}^\intercal u \leq b_{j}\}$ are polytopic sets defined as intersecting hyperplanes, and where $E_k = [0_{n \times nk},I_n,0_{n(N-k)\times n}]$, and $F_k = [0_{m \times mk},I_m,0_{m(N-k-1) \times m}]$ isolate the $k^{\mathrm{th}}$ element of the state and input, respectively, and $\Delta_X,\Delta_U \in [0,1)$ are constraint violation thresholds.

Using \eqref{eq:concatdyn}, the cost in \eqref{eq:COST} can be re-written as
\begin{equation}
\label{eq:concatCost}
    J(\bfU) = \E\left[(\bfX - X_d)^\intercal \mathcal{Q} (\bfX - X_d) + \bfU^\intercal \mathcal{R} \bfU\right],
\end{equation}
with $\mathcal{Q}=\textrm{diag}\left(Q_0,\ldots,Q_{N-1}\right)$, $\mathcal{R}=\textrm{diag}\left(R_0,\ldots,R_{N-1}\right)$, and $X_d = \left[x_{d,0}\ \cdots\ x_{d,N}\right]^\intercal \in\R^{(N+1)n}$.
Since $Q_k \succeq 0$ and $R_k \succ 0$, $\forall k \in \N_{[0,N-1]}$, it follows that $\mathcal{Q} \succeq 0$ and $\mathcal{R} \succ 0$.

\begin{problem}
\label{prob:stoc}
Solve the optimization problem
\begin{subequations}
    \begin{align}
        \underset{\bfU}{\mathrm{minimize}} &\;\;   J(\bfU),  \\
        \mathrm{subject\ to}&\;\; \eqref{eq:concatdyn},\ \eqref{eq:sysCC}, \ \mathrm{and} \\
        &E_0 \bfX \sim \psi_{\bfx_0}, \ E_N \bfX \sim \psi_{\bfx_f}, \ \bfW \sim \psi_{\bfW}.
    \end{align}
\end{subequations}%
\end{problem}%
The goal of Problem \ref{prob:stoc} is to minimize the quadratic cost \eqref{eq:concatCost}, satisfy the constraints in \eqref{eq:concatdyn}, \eqref{eq:sysCC}, while steering the state of (\ref{eq:discreteDynamics}) from the given initial distribution $\psi_{\bfx_0}$ to the desired terminal state distribution $\psi_{\bfx_f}$.

\subsection{Characteristic Functions}

One way to represent the underlying system stochasticity is via characteristic functions (CF). 
\begin{defn}[]
For a continuous random vector $\bfw\in\mathbb{R}^{p}$ such that $\bfw \sim \psi_{\mathbf{w}}$, the CF is defined by the Fourier transform $\mathcal{F}\{\psi_{\bfw}\}(t)$ 
of its pdf,
\begin{align}
    \upvarphi_{\bfw}(t) = \mathbb{E}_{\bfw}[\exp(\i t^\intercal\bfw)]
    = \int_{\mathbb{R}^p}e^{\i t^\intercal z} \psi_{\bfw}(z)\, \d z,
\end{align} 
where $t,z\in\mathbb{R}^{p}$.
\end{defn}
The CF  has the following properties~\cite{lukacs_characteristic_1970,ushakov_selected_1999}: 
\begin{itemize}
    \item It is uniformly continuous.
    \item $\upvarphi_{\bfw}(0) = 1$.
    \item It is bounded, i.e., $|\upvarphi_{\bfw}(t)| \leq 1,$ for all $t \in \R^{p}$.
    \item It is Hermitian, i.e., $\upvarphi_{\bfw}(-t) = \overline{\upvarphi}_{\bfw}(t)$.
\end{itemize}

\begin{assum}
The CF $\upvarphi_{\bfw}$ is absolutely integrable, that is, it is an element of $L_1(\Re^p)$.
\end{assum}

To recover the pdf from its CF, we use the following result. 

\begin{thm}[Inversion Theorem for pdfs, {\cite[Theorem 1.2.6]{ushakov_selected_1999}}]
If the CF $\upvarphi_{\bfw} \in L_1(\mathbb{R}^p)$, then the pdf can be recovered via the inverse Fourier transform $\mathcal{F}^{-1}\{\upvarphi_{\bfw}\}(z)$,

\begin{equation}
    \psi_{\bfw}(z) =  \left(\frac{1}{2\pi}\right)^{p}\int_{\R^{p}}e^{-\mathrm{i} t^\intercal z} \upvarphi_{\bfw}(t) \, \d t.
    \label{eq:inv_FT}
\end{equation}
\end{thm}

Below, we summarize useful properties of CFs.
Let $\bfw_1,\bfw_2,\bfw,\mathbf{z}$ be random vectors of appropriate dimensions.

\begin{enumerate}
	\item 
	If $\mathbf{z} = \bfw_1 + \bfw_2$, then $\psi_{\mathbf{z}}(z) = \big(\psi_{\bfw_1} * \psi_{\bfw_2}\big)(z)$ (i.e., convolution of their pdfs), and $\upvarphi_{\mathbf{z}}(t) = \upvarphi_{\bfw_1}(t)\upvarphi_{\bfw_2}(t)$ \cite[Sec. 21.11]{cramer_mathematical_1999}.
	
	\item 
	If $\bfz = F\bfw + g$ for $F \in \mathbb{R}^{n\times p}, \ g\in\R^{n}$, then $\upvarphi_{\bfz}(t) = \exp(\i t^\intercal g)\upvarphi_{\bfw}(F^\intercal t)$ \cite[Sec. 22.6]{cramer_mathematical_1999}.
	
	\item 
	Given $\bfw_1$ and $\bfw_2$, then $\bfz = [\bfw_1^\intercal, \bfw_2^\intercal]^\intercal$ has the pdf $\psi_{\bfz}(z) = \psi_{\bfw_1}(e_1^\intercal z_1)\psi_{\bfw_2}(e_2^\intercal z_2), z = [z_1^\intercal, z_2^\intercal]^\intercal$, and CF $\upvarphi_{\bfz}(t) = \upvarphi_{\bfw_1}(e_1^\intercal t)\upvarphi_{\bfw_2}(e_2^\intercal t)$, $t=[t_1^\intercal,t_2^\intercal]^\intercal$, 
	where $e_1$ and $e_2$ isolate the first and second component of the vector, respectively \cite[Sec. 22.4]{cramer_mathematical_1999}.
	
	\item 
	If $\bfz = [\bfz_1 ~ \cdots ~\bfz_i~ \cdots ~\bfz_p]^\intercal \in \mathbb{R}^p$ is a vector of scalar random variables $\mathbf{z}_i$ 
	with pdfs
	$\psi_{\bfz_i},$ then the pdf of $a^\intercal\bfz$,   $a\in\mathbb{R}^p$, is $\psi_{a^\intercal\bfz}(z) = \prod_{i=1}^{p}\psi_{\bfz_i}(e_{i,p}^\intercal a z)$, and the CF is $\upvarphi_{a^\intercal\bfz}(t) = \upvarphi_{\bfz}(t_{\bfz}) =  \prod_{i=1}^{p}\upvarphi_{\bfz_i}(t_i)$, for $t_{\bfz} = a t$, and $t_{i} = e_{i,p}^\intercal t_{\bfz}$ \cite[Sec. 22.4]{cramer_mathematical_1999}.
\end{enumerate}

We can also recover the cdf via the CF using the following theorem. 
\begin{thm}[Gil-Pelaez Inversion Theorem,\cite{gil-pelaez_note_1951,ushakov_selected_1999}]
\label{thm:GPI}
Given a random variable $\mathbf{y}$ with CF $\upvarphi_{\mathbf{y}}$ and pdf $\psi_{\bfy}$ satisfying the property that $\int_{\mathbb{R}} \log(1+|x|) \psi_{\bfx}(z) \d z < \infty$, 
then the cdf of $\mathbf{y}$, $\Phi_{\mathbf{y}}$, at each point $y$ that is continuous, can be evaluated by
\begin{equation}\label{eq:cftocdf}
         \Phi_{\mathbf{y}}(y) = \frac{1}{2}-\frac{1}{\pi}\int^{\infty}_{0}\frac{1}{t}\mathrm{Im}\left[\exp{(-\mathrm{i}t y)} \,\upvarphi_{\mathbf{y}}(t)\right] \, \d t,
\end{equation}
where $y,t\in\mathbb{R}$. 
\end{thm}

\begin{rem}
The requirement $\int_{\mathbb{R}} \log(1+|x|) \psi_{\bfx}(z) \d z < \infty$ is a mild condition which is satisfied by many distributions~\cite{witkovsky_numerical_2016}, including those used in this paper.  
\end{rem}

\subsection{Proposed Distribution Steering Controller}

The following assumption outlines the controller structure used in this work.

\begin{assum}[Feedback law]
The controller in \eqref{eq:concatdyn} has an affine state feedback structure, given by
\begin{equation}
    \bfu_k = \sum_{i=0}^{k} L_{k,i} \bfx_i + g_k. \label{eq:feedback}
\end{equation}
Concatenating these vectors yields $\bfU = L \bfX + g$, where $L\in\R^{Nm \times (N+1)n}$ is a lower block triangular matrix and $g = [g_0^\intercal,\ldots,g_{N-1}^\intercal]^\intercal\in\R^{Nm}$.
\end{assum}

Note that this feedback law uses the full state \textit{history} to determine the control input at every time step $k$, as opposed to just using the current state $\bfx_k$.

\begin{prop}[Affine disturbance feedback \cite{okamoto2018optimal}]
    The feedback law in (\ref{eq:feedback}) results in the state and input sequences
    \begin{subequations}
    \label{eq:newStateControl}
        \begin{align}
            \bfX &= (I - \scriptB L)^{-1}(\scriptA \bfx_0 + \scriptD \bfW + \scriptB g), \label{eq:newdyn}\\
            \bfU &= L(I - \scriptB L)^{-1}(\scriptA \bfx_0 + \scriptD \bfW + \scriptB g) + g. \label{eq:newinput}
        \end{align}
    \end{subequations}
\end{prop}
\begin{proof}
    Plugging (\ref{eq:feedback}) into \eqref{eq:concatdyn} yields (\ref{eq:newdyn}).
    Similarly, plugging (\ref{eq:newdyn}) into (\ref{eq:feedback}) yields (\ref{eq:newinput}).
\end{proof}

\begin{corr}
Given the affine disturbance feedback terms, 
\begin{subequations}
\label{eq:newController}
    \begin{align}
        K &= L(I - \scriptB L)^{-1}, \label{eq:Kdef} \\
        v &= L(I - \scriptB L)^{-1} \scriptB g + g, \label{eq:vdef}
    \end{align}
\end{subequations}
then the state and input sequences in (\ref{eq:newStateControl}) can be equivalently written as 
    \begin{subequations}
    \label{eq:newnewStateControl}
        \begin{align}
            \bfX &= (I + \scriptB K)(\scriptA \bfx_0 + \scriptD \bfW) + \scriptB v, \label{eq:newnewdyn}\\
            \bfU &= K(\scriptA \bfx_0 + \scriptD \bfW) + v, \label{eq:newnewinput}
        \end{align}
    \end{subequations}
\end{corr}
\begin{proof}
    We substitute \eqref{eq:Kdef} and \eqref{eq:vdef} into \eqref{eq:newinput}, and substitute $\bfU$ into \eqref{eq:concatdyn},
    to obtain
    \begin{subequations}
    \begin{align}
        \bfX &= \scriptA \bfx_0 + \scriptB \left(L(I - \scriptB L)^{-1}(\scriptA \bfx_0 + \scriptD \bfW + \scriptB g) + g\right) \nonumber\\
        & \hspace{17em}+ \scriptD \bfW,  \label{eq:dyn_middle}\\
    \bfU    &= (I+\scriptB L(I - \scriptB L)^{-1})(\scriptA \bfx_0 + \scriptD \bfW)\nonumber \\
        & \hspace{10.3em} + L(I - \scriptB L)^{-1} \scriptB g + g.\label{eq:control_middle}
    \end{align}
    \end{subequations}
    Simplifying \eqref{eq:dyn_middle} and \eqref{eq:control_middle} yields the desired result \eqref{eq:newnewdyn} and \eqref{eq:newnewinput}.
\end{proof}

\section{Chance Constraints with Affine Feedback via Characteristic Functions}\label{sec:chanceconstraints}

In this section, we present a decomposition of the chance constraints using Boole's inequality via
affine disturbance feedback and represent the probabilistic constraints (also known as chance constraints) using CFs.
The resulting constraint is an integral transform over the linear system and the polytopic constraints.


\subsection{Reformulation of Chance Constraints}

\subsubsection{State Chance Constraints}

The joint state chance constraints in (\ref{eq:stateCC}) can be transformed into a set of \textit{individual} chance constraints through Boole's inequality \cite{ono2008iterative,farina_stochastic_2016,mesbah2016stochastic}
\begin{align}
    \hspace{-0.5em}\P\{\alpha_{j,k}^{\intercal} E_k \mathbf{X} \leq \beta_{j,k}\} \geq 1 - \delta^x_{j,k}, \quad \sum_{k=1}^{N}\sum_{j=1}^{N_X} \delta^x_{j,k} \leq \Delta_X,
\end{align}  
equivalently,
\begin{align} \label{eq:indivChanceConstraints}
\hspace{-0.5em}\Phi_{\alpha_{j,k}^{\intercal} E_k \mathbf{X}}\left(\beta_{j,k}\right) \geq 1 - \delta^x_{j,k}, \quad \sum_{k=1}^{N}\sum_{j=1}^{N_X} \delta^x_{j,k} \leq \Delta_X,
\end{align}
where $\delta^x_{j,k} \in [0,\Delta_X),\  k\in\mathbb{N}_{[1,N]},\  j\in\mathbb{N}_{[1,N_X]}$.
Plugging \eqref{eq:newnewdyn} into \eqref{eq:indivChanceConstraints} yields
\begin{align}
\Phi_{\mathbf{c}_{j,k}}(\beta_{j,k} - \alpha_{j,k}^\intercal E_k \scriptB v)\geq 1 - \delta^x_{j,k}, \label{eq:FullStateCC}
\end{align}
where $\bfc_{j,k}= \alpha_{j,k}^{\intercal} E_k (I + \scriptB K) (\scriptA \bfx_0 + \scriptD \bfW)$.

\subsubsection{Input Chance Constraints}

Similar to the state chance constraints, we transform the joint chance constraints in (\ref{eq:inputCC}) into a set of individual chance constraints as follows
    \begin{equation}
        \P\{a_j^{\intercal} F_k \mathbf{U} \leq b_j\} \geq 1 - \delta_{j,k}^{u}, \quad
        \sum_{k=1}^{N}\sum_{j=1}^{N_U} \delta_{j,k}^{u} \leq \Delta_U.
        \end{equation}
        Equivalently,
    \begin{align} \label{eq:chanceConstraintsInput}
  \Phi_{a_j^{\intercal} F_k \mathbf{U}}(b_j) \geq 1 - \delta_{j,k}^{u}, \quad
    \sum_{k=1}^{N}\sum_{j=1}^{N_U} \delta_{j,k}^{u} \leq \Delta_U,
    \end{align}
where $\delta_{j,k}^{u}\in[0,\Delta_U), \  k\in\mathbb{N}_{[0,N-1]}, \  j\in\mathbb{N}_{[1,N_U]}$.
Using (\ref{eq:newnewinput}) in the chance constraints (\ref{eq:chanceConstraintsInput}) yields 
\begin{equation}
    \Phi_{\bfd_{j,k}}(b_j - a_j^\intercal F_k v) \geq 1 - \delta_{j,k}^{u}, \label{eq:fullInputCC}
\end{equation}
where $\bfd_{j,k} = a_j^\intercal F_k K(\scriptA \bfx_0 + \scriptD \bfW)$.

\subsection{Encoding Chance Constraints in the Presence of Affine Feedback}
 
To illustrate how to encode the chance constraints (\ref{eq:FullStateCC}) and (\ref{eq:fullInputCC}) via CFs, consider first  the state chance constraint in (\ref{eq:FullStateCC}). 
Expanding the random variable $\bfc_{j,k}$ we can write
\begin{align}
	\bfc_{j,k} = \mu_{\bfc,j,k}^\intercal \bfx_0 + \nu_{\bfc,j,k}^\intercal \bfW,
\end{align}
where $\mu_{\bfc,j,k}^\intercal = \alpha_{j,k}^\intercal E_k (I + \mathcal{B} K)\mathcal{A}$ and $\nu_{\bfc,j,k}^\intercal = \alpha_{j,k}^\intercal E_k (I + \mathcal{B} K)\scriptD$ are non-random variables that are linear in the decision variable $K$.
Under Assumption~1, Property 3 of CFs allows us to 
decompose $\upvarphi_{\bfc_{j,k}}$ as 
\begin{equation}
\label{eq:decomp1}
    \upvarphi_{\bfc_{j,k}}(t) = \upvarphi_{\mu_{\bfc,j,k}^\intercal\bfx_0}(t)\upvarphi_{\nu_{\bfc,j,k}^\intercal\bfW}(t). 
\end{equation}
Next, Property 2 yields
\begin{subequations}
\label{eq:decomp2}
    \begin{align}
        \upvarphi_{\mu_{\bfc,j,k}^\intercal\bfx_0}(t) &= \upvarphi_{\bfx_0}(t_{\bfx_0}^{\bfc}), \label{eq:decomp2a} \\
        \upvarphi_{\nu_{\bfc,j,k}^\intercal\bfW}(t) &= \upvarphi_{\bfW}(t_{\bfW}^{\bfc}). \label{eq:decomp2b}
    \end{align}
\end{subequations}
where $t_{\bfx_0}^{\bfc} = \mu_{\bfc,j,k} t$ and $t_{\bfW}^{\bfc} = \nu_{\bfc,j,k} t$.
Finally, Property 4 yields
\begin{subequations}
\label{eq:decomp3}
    \begin{align}
        \upvarphi_{\bfx_0}(t_{\bfx_0}^{\bfc}) &= \prod_{i=1}^{n}\upvarphi_{\bfx_{0,i}}(t_{\bfx_{0,i}}^{\bfc}), \label{eq:decomp3a} \\
        \upvarphi_{\bfW}(t_{\bfW}^{\bfc}) &= \prod_{i=1}^{pN} \upvarphi_{\bfW_{i}}(t_{\bfW_i}^{\bfc}). \label{eq:decomp3b}
    \end{align}
    where $t_{\bfx_{0,i}}^{\bfc} = e_{i,n}^\intercal t_{\bfx_0}^{\bfc}$ and $t_{\bfW_i}^{\bfc} = e_{i,pN}^\intercal t_{\bfW}^{\bfc}$. 
\end{subequations}%
Theorem 1 gives an analytical expression for the cdf of $\bfc_{j,k}$ evaluated at $\gamma_{j,k} = \beta_{j,k} - \alpha_{j,k}^\intercal E_k \scriptB v$ as 
\begin{align}
	\Phi_{\bfc_{j,k}}(\gamma_{j,k}) =
	\frac{1}{2} - \frac{1}{\pi}\int_{0}^{\infty} \frac{1}{t}\mathrm{Im}\left(e^{-\mathrm{i} t\gamma_{j,k}}\upvarphi_{\bfc_{j,k}}(t)\right) \d t, \label{eq:stateCCencoding}
\end{align}
where $\upvarphi_{\bfc_{j,k}}(t)$ is given in \eqref{eq:decomp1}.
Thus, (\ref{eq:stateCCencoding}) provides a means to compute the cdf in the state chance constraints \eqref{eq:FullStateCC}, and encodes the decision variables $K$ and $v$ through Theorem 1.

Similarly, we can also derive the input chance constraints in \eqref{eq:fullInputCC} for the random variable $ \bfd_{j,k}$ by rewriting
\begin{equation}
    \bfd_{j,k} = \mu^\intercal_{\bfd,j,k} \bfx_0 + \nu^\intercal_{\bfd,j,k} \bfW,
\end{equation}
where $\mu^\intercal_{\bfd,j,k} = a_j^\intercal F_k K\scriptA$ and $\nu^\intercal_{\bfd,j,k} = a_j^\intercal F_k K \scriptD$.
From Property 3 of CFs, the CF of $\bfd_{j,k}$ is therefore
\begin{equation}
    \upvarphi_{\bfd_{j,k}}(t) = \upvarphi_{\mu_{\bfd,j,k}^\intercal\bfx_0}(t)\upvarphi_{\nu_{\bfd,j,k}^\intercal\bfW}(t). \label{eq:decomp4}
\end{equation}
Further, by Properties 2 and 4, we have
\begin{subequations}
\begin{align}
    \upvarphi_{\mu_{\bfd,j,k}^\intercal\bfx_0}(t) &= \upvarphi_{\bfx_0}(t_{\bfx_0}^{\bfd}) = \prod_{i=1}^{n}\upvarphi_{\bfx_{0,i}}( t_{\bfx_{0,i}}^{\bfd}), \label{eq:decomp5a} \\
    \upvarphi_{\nu_{\bfd,j,k}^\intercal\bfW}(t) &= \upvarphi_{\bfW}(t_{\bfW}^{\bfd}) = \prod_{i=1}^{pN} \upvarphi_{\bfW_{i}}(t_{\bfW_i}^{\bfd}) \label{eq:decomp5b}
\end{align}
where $t_{\bfx_0}^{\bfd} = \mu_{\bfd,j,k} t,$  $t_{\bfW}^{\bfd} = \nu_{\bfd,j,k} t$, $t_{\bfx_{0,i}}^{\bfd} = e_{i,n}^\intercal t_{\bfx_0}^{\bfd}$, 
and $t_{\bfW_i}^{\bfd} = e_{i,pN}^\intercal t_{\bfx_0}^{\bfd}$.
\end{subequations}
Thus, the expression for the cdf of $\bfd_{j,k}$ evaluated at $\gamma_{\bfd,j,k} = b_{j} - a_{j}^\intercal F_k v$ is given by 
\begin{align}
	\Phi_{\bfd_{j,k}}(\gamma_{j,k}) =
	\frac{1}{2} - \frac{1}{\pi}\int_{0}^{\infty} \frac{1}{t}\mathrm{Im}\left(e^{-\mathrm{i} t\gamma_{\bfd,j,k}}\upvarphi_{\bfd_{j,k}}(t)\right) \d t. \label{eq:inputCCencoding}
\end{align}
Note that the constraints encoded by the CF in \eqref{eq:stateCCencoding} and \eqref{eq:inputCCencoding} result in \textit{nonlinear} constraints in terms of the decision variables $K$ and $v$.


\section{Terminal Density Constraints}\label{sec:terminal_density_constraints}

Our approach aims at  matching probability densities using the machinery of characteristic functions. 
The benefit of using characteristic functions to match between densities as opposed to other metrics such as KL-divergence or Wasserstein distance is two-fold.
First, it can be shown that the largest absolute difference between two pdfs is bounded by the $L_1$ difference of their CFs. 
Second, this holds for all distributions (including mixture distributions) which have a CF in $L_1(\mathbb{R}^n)$ and directly results in an explicit integral expression over the frequency domain, and not an integration over the entire state-space~\cite{balci2020covariance}.
This is convenient, as it is difficult to formulate the terminal constraints analytically in the state-space due to the non-Gaussian state evolution requiring several convolutions at each 
time step (see Property 1 of operations on CFs). 

Next, we first derive a joint distribution representation which results in an $n$-dimensional integral. 
We then show that due to the independence property of the disturbances, we can compute this integral using $n$ separate matching constraints at the final time.


\subsection{Joint CF Representation of the Terminal Density}

Using Properties 1 and 2 of operations on CFs, the joint CF of the terminal state $\bfx_{N}$ is 
    \begin{align}
        \label{eq:ENX_CF}
        \upvarphi_{\bfx_{N}}(t) &= \prod_{i=1}^{n} \upvarphi_{\bfx_{N},i}(t_i) \nonumber \\
        &=\prod_{i=1}^{n} \exp(\i \sigma_i^\intercal t)\upvarphi_{\bfx_{0}}(t_{\bfx_{0}}^{N})\upvarphi_{\bfW}(t_{\bfW}^{N}), 
        \end{align}
where
\begin{align}
    \upvarphi_{\bfx_0}(t_{\bfx_0}^{N}) &= \prod_{j=1}^{n}\upvarphi_{\bfx_{0,j}}(t_{\bfx_{0,j}}^{N}), \label{eq:decomp6a} \\
    \upvarphi_{\bfW}(t_{\bfW}^{N}) &= \prod_{j=1}^{pN} \upvarphi_{\bfW_{j}}(t_{\bfW_j}^{N}), \label{eq:decomp6b}
\end{align}
and where
$t_{\bfx_0}^{N} = \mu_i t$, $t_{\bfx_0,j}^{N} = e_{j,n}^\intercal t_{\bfx_0}^{N}, t_{\bfW}^{N} = \nu_i t$, $t_{\bfW_j}^{N} = e_{j,pN}^\intercal t_{\bfW}^{N}$, and $\mu_i^\intercal = e_{i,n}^\intercal E_N(I + \scriptB K)\scriptA, \ \nu_i^\intercal = e_{i,p}^\intercal E_N(I + \scriptB K)\scriptD$, and $\sigma_i^\intercal = e_{i,m}^\intercal E_N \scriptB v$. 
Similarly, the joint CF of the \textit{desired} terminal state is 
 \begin{align}       
    \upvarphi_{\bfx_f}(t) &= \prod_{i=1}^{n} \upvarphi_{\bfx_{f,i}}(t_i).
    \label{eq:Xd_CF}
\end{align}

We now introduce the $L_1$ distance as an upper bound on the maximum $L_1$ deviation between two probability distributions.

\begin{thm}[{\cite[Sec. 1.4]{ushakov_selected_1999}}]
\label{thm:joint_matching}
If the joint pdf for the terminal state of the system is $\psi_{\bfx_{N}}$ with CF \eqref{eq:ENX_CF} and the desired joint pdf is $\psi_{\bfx_f}$ with CF \eqref{eq:Xd_CF}, then 
\begin{align}
    \label{eq:joint_matching_1}
\Delta_{\psi_{\bfx_N}}(K,v) =&\sup_{z\in\mathbb{R}^n} | \psi_{\bfx_{N}}(z; K,v) -  \psi_{\bfx_f}(z)| \leq D(K,v),
\end{align}%
where 
\begin{align}
    \label{eq:joint_matching_2}
    D(K,v) &= \left(\frac{1}{2\pi}\right)^n \|\upvarphi_{\bfx_{N}} - \upvarphi_{\bfx_{f}}\|_{1} \nonumber \\
    &=\left(\frac{1}{2\pi}\right)^n \int_{\mathbb{R}^n}\left|\upvarphi_{\bfx_{N}}(t) -  \upvarphi_{\bfx_f}(t)\right| \d t. 
\end{align}
\end{thm}

\begin{proof}
See Appendix A.
\end{proof}

\begin{corr}
Let $\epsilon > 0$ such that $D(K,v)< \epsilon$ for some $K$ and $v$.
Then, $\sup_{z\in\mathbb{R}^n}\left|\psi_{\bfx_{N}}(z; K,v) -  \psi_{\bfx_f}(z)\right|\leq \epsilon$.
\end{corr}

\subsection{Matching Densities}

We derive a simpler representation of the $n-$dimensional integral of the joint representation in \eqref{eq:joint_matching_1}. 
Specifically, we construct $n$ separate density matching expressions with respect to each terminal state variable, that is, for all $i \in \mathbb{N}_{[1,n]}$,
\begin{align}
    \label{eq:joint_matching_indiv}
&\Delta_{\psi_{\bfx_N},i}(K,v)=\sup_{z_i\in\mathbb{R}} \left| \psi_{\bfx_{N},i}(z_i; K,v) -  \psi_{\bfx_{f,i}}(z_i)\right|\nonumber \\
&\hspace{16em} \leq D_i(K,v),
\end{align}%
where 
\begin{equation}
\label{eq:marginal_matching}
        D_i(K,v) = \frac{1}{2\pi} \int_{\mathbb{R}} \left|\upvarphi_{\bfx_{N},i}(t_i) -  \upvarphi_{\bfx_f,i}(t_i)\right| \d t,
\end{equation}
which follows from Theorem \ref{thm:joint_matching}.
The next result provides
 a relationship between \eqref{eq:joint_matching_2} and \eqref{eq:marginal_matching}. 

\begin{thm}  \label{Theorem4}
Suppose there exists $(K,v)$ such that, for all $i\in \mathbb{N}_{[1,n]}$, $D_i(K,v)\leq \epsilon_i$.
Then, $D(K,v) \leq (1/2\pi)^{n-1}\epsilon$, where $\epsilon = \sum_i\epsilon_i$. 
\end{thm}

\begin{proof}
See Appendix B. 
\end{proof}

\section{Resulting optimization problem}\label{sec:resulting_optimization}

With the elements derived for both the chance constraints and the terminal distribution constraint, we present the resulting optimization problem. 


\begin{problem}
Solve the optimization problem
    \begin{subequations}
    \label{opt:resulting_optimization}
    \begin{align}
        \underset{K,v,\delta^x,\delta^u}{\mathrm{min}} &\;\; J(K,v) + \sum_{i=1}^{n} \lambda_i D_i(K,v) \label{eq:fullCost} \\
        \mathrm{s.t.}&\;\; \Phi_{\mathbf{c}_{j,k}}(\beta_{j,k} - \alpha_{j,k}^\intercal E_k \scriptB v) \geq 1 - \delta^x_{j,k},\\
        &\;\; \Phi_{\bfd_{j,k}}(b_j - a_j^\intercal F_k v) \hspace{2.1em} \geq 1 - \delta_{j,k}^{u}, \\
        &\;\; \sum_{k=1}^{N}\sum_{j=1}^{N_X} \delta^x_{j,k} \leq \Delta_X,\hspace{1em} \sum_{k=1}^{N}\sum_{j=1}^{N_U} \delta_{j,k}^{u} \leq \Delta_U,
    \end{align}
\end{subequations}%
where $J(K,v)$ is given by
\begin{align} 
     &\left[(I + \scriptB K)(\scriptA \mathbb{E}[\bfx_0] + \scriptD \mathbb{E}[\bfW]) + \scriptB v - X_d\right]^\intercal \scriptQ\cdot \nonumber \\
    &\quad\,\left[(I + \scriptB K)(\scriptA \mathbb{E}[\bfx_0] + \scriptD \mathbb{E}[\bfW]) + \scriptB v - X_d \right] \nonumber \\
    &+ \left[K(\scriptA \mathbb{E}[\bfx_0] + \scriptD \mathbb{E}[\bfW]) + v\right]^\intercal \scriptR \nonumber \\
    &+ \mathrm{tr}\left[\big((I + \scriptB K)^\intercal \scriptQ (I + \scriptB K) + K^\intercal \scriptR K\big)\Sigma\right] \label{eq:newCost},
\end{align}
where  $\Sigma = \scriptA \Sigma_{\bfx_0} \scriptA^\intercal + \scriptD \Sigma_{\bfW} \scriptD^\intercal$. 
\end{problem}

We treat the matching constraint as a soft constraint, as in~\cite{balci2020covariance}. 
By penalizing this $L_1$ distance in the cost, we provide flexibility to the underlying nonlinear program solver and enable increased feasibility.

\section{Numerical Example - 2D Double Integrator}\label{sec:examples}

We demonstrate our approach on a 2D double integrator with different disturbances and initial conditions.
Consider the system (\ref{eq:discreteDynamics}) 
with state $x = [x\ \dot x \ y\ \dot y]^\intercal$. 
The expressions for the system matrices $A_k,\ B_k,$ and $D_k$ are given in \cite{JoshJack}
with $\Delta T = 1$ and $N = 5$.
We assume polytopic state constraints $\mathcal{X}_k$, with
$\alpha_{1k} =     
    \begin{bmatrix}
    1 & 1 & 0 & 0\\
    \end{bmatrix},\
    \beta_{1k} = 
    12.75,\ 
    \alpha_{2k} =     
    \begin{bmatrix}
    1 & 0.1 & 0 & 0\\
    \end{bmatrix},\
    \beta_{2k} = 
    8.75$
for $k\in\mathbb{N}_{[1,N]}$, and assume $\bfx_0\sim\psi_{x_0}$ must be within the state polytopic constraints, as well.
We let $\mathcal{U}_k = [-4,4]^2$ for $k\in\mathbb{N}_{[0,N-1]}$.
The desired trajectory $X_d$ is interpolated from waypoints $(4,5)$ to $(8,5)$ for $k\in\mathbb{N}_{[0,2]}$ and from $(9,5)$ to $(7.875,3)$ for $k\in\mathbb{N}_{[3,5]}$. 
We seek to drive the final state to $\bfx_{N} \sim \psi_{\bfx_f} = \mathcal{N}(\mu_{\bfx_f},\Sigma_{\bfx_f})$ with mean 
$\mu_{\bfx_f} = [7.75\ 2\ 0\ 0]^\intercal$ and variance $\Sigma_{\bfx_f} = \mathrm{diag}([0.06\ 0.006\ 0.6\ 0.006])$. 
We choose $\Delta_X = 0.1$ and $\Delta_U = 0.1$, and 
$Q_k = \mathrm{diag}([10\ 1\ 10\ 1])$, $R_k = \mathrm{diag}([1\ 1])$ for $i\in\mathbb{N}_{[0,N-1]}$. The weighting of the distance metrics in the cost is $\lambda = [10\ 1\ 10\ 1]^\intercal$.

All computations were done in MATLAB with an Intel Core i9-10900K processor and 64GB RAM.
The optimization problems were solved using \texttt{fmincon}. 
The CF inversion \eqref{eq:stateCCencoding} uses CharFunTool \cite{witkovsky_numerical_2016} and the density matching constraint in \eqref{eq:marginal_matching} was implemented using trapezoidal quadrature.
We used $10^4$ Monte-Carlo samples to verify average state and input constraint violation (denoted as $\Delta_{\rm X,MC}$ and $\Delta_{\rm U,MC}$, respectively) and cost (denoted as $J_{\rm MC}(K,v)$). 

\subsection{Standard Gaussian Distribution}
\label{subsec:gauss}

To validate our approach, we first considered $\bfx_0\sim \psi_{\bfx_0} = \mathcal{N}(\mu_{\bfx_0},\Sigma_{\bfx_0})$
with $\mu_{\bfx_0} = [4\ 0\ 5\ 0]^\intercal$ and $\Sigma_{\bfx_0} = \mathrm{diag}([0.18\ 0.002\ 0.18\ 0.002])$;
the disturbance is $\mathcal{N}(\mu_{\bfW},\Sigma_{\bfW})$ with mean $\mu_{\bfW} = [0\ 0]^\intercal$ and variance $\Sigma_{\bfW} = \mathrm{diag}([1\ 1])$ for the entire horizon. 
As shown in Figure \ref{fig:gaussian2Dstate}, our method drives the system to follow the reference trajectory, while not significantly violating the state constraints (Table~\ref{tab:cost_violation}).
Likewise, input violation is minimal, as shown in Figure~\ref{fig:gaussian2Dinput} and 
Table~\ref{tab:cost_violation}. 
Since we steer the state from an initial Gaussian distribution to a final Gaussian distribution, the maximum deviation between the final and desired pdfs and the corresponding $L_1$ distances are small (Table \ref{tab:pdf_deviation}). 

\subsection{Heavy Tail - Laplace Distribution}
\label{subsec:laplace}

Heavy-tailed distributions are of interest as they decay much more slowly than Gaussians, but with a similar mean and variance. 
The Laplace distribution has the pdf $\mathcal{L}_{\mu,\beta}(x) =  \exp\left(-{|x-\mu|}/{\beta}\right)/ 2\beta$.
We assume the initial condition $\bfx_0 \sim \mathcal{L}(\mu_{\bfx_0},\beta_{\bfx_0})$
with location $\mu_{\bfx_0} = [4\ 0\ 5\ 0]^\intercal$ and scale $\beta_{\bfx_0} = [0.3\ 0.01\ 0.3\ 0.01]^\intercal$. 
The disturbance also follows a Laplace distribution $\bfW\sim\mathcal{L}(\mu_{\bfW},\beta_{\bfW})$ with location
$\mu_{\bfW} = [0\ 0]^\intercal$ and scale $\beta_{\bfW} = [1\ 1]^\intercal$.
Although the Laplace distribution is not smooth (Figure \ref{fig:laplace2Dstate}), 
our method is able to steer to the final desired density with little constraint violation (Table \ref{tab:cost_violation}). 
The input in Figure \ref{fig:laplace2Dinput} shows that there is some violation of the bounds, but it is within the violation threshold (Table \ref{tab:cost_violation}).
The larger deviation between the final and desired pdfs (Table \ref{tab:pdf_deviation}) reflects the fact that we modify a random variable that is not Gaussian so as to behave like a Gaussian one.

\subsection{Mixture Distributions - Normal Mixture}
\label{subsec:multiGauss}

Lastly, we consider a Gaussian mixture with $\bfx_0\sim \psi_{\bfx_0} = 0.5\mathcal{N}(\mu_{\bfx_{0,1}},\Sigma_{\bfx_{0,1}}) + 0.5\mathcal{N}(\mu_{\bfx_{0,2}},\Sigma_{\bfx_{0,2}})$ with means $\bfx_{0,1} = [4\ 0\ 5\ 0]^\intercal$, $\bfx_{0,2} = [3.5\ 0.1\ 3.5\ 0.1]^\intercal$ and covariances $\Sigma_{\bfx_{0,1}} = \mathrm{diag}([0.3\ 0.01\ 0.3\ 0.01])$, $\Sigma_{\bfx_{0,2}} = \mathrm{diag}([0.1\ 0.01\ 0.5\ 0.01])$. 
The disturbance is a Gaussian mixture, $\bfW\sim \psi_{\bfW} = 0.5\mathcal{N}(\mu_{\bfW_{1}},\Sigma_{\bfW_1}) + 0.5\mathcal{N}(\mu_{\bfW_{2}},\Sigma_{\bfW_{2}})$, with means $\mu_{\bfW_{1}} = [0\ 0.1]^\intercal$, $\mu_{\bfW_{2}} = [0.1\ 0]^\intercal$ and covariances $\Sigma_{\bfW_{1}} = \mathrm{diag}([1\ 1])$, $\Sigma_{\bfW_{2}} = \mathrm{diag}([1\ 1])$.
The affine controller steers the Gaussian mixture to a single Gaussian (Figure \ref{fig:multiGaussian2Dstate}) with minimal violation of the state constraints (Table \ref{tab:cost_violation}).
The input remains within acceptable limits (Table \ref{tab:cost_violation}) despite the multi-modal nature of the noise (Figure \ref{fig:multiGaussian2Dinput}).
The deviation between final and desired pdfs are much smaller than seen with the Laplace pdf (Table~\ref{tab:pdf_deviation}).
This is likely because 
the controller alters the weights of the multi-modal Gaussian elements to match the desired, final Gaussian density.

\section{Conclusions}\label{sec:conclusion}

We have formulated a tractable solution of the distribution steering problem under general, not necessarily Gaussian,  disturbances. 
We showed that using Boole's inequality and characteristic functions, we can turn the problem into a nonlinear optimization problem.
Future work will aim to further utilize the structure of the CFs to obtain faster, real-time solutions and extend the approach to nonlinear systems.

\section*{Acknowledgement}\label{sec:ack}

We thank Jack Ridderhof for several discussions and Adam Thorpe for providing Figure \ref{fig:dist_steering_visual}.
This work has been supported in part by the National Science Foundation under award CNS-1836900 and
by NASA under the University Leadership Initiative award \#80NSSC20M0163.
Any opinions, findings, and conclusions or recommendations expressed in this material are those of the authors and do not necessarily reflect the views of the NSF or any NASA entity.

\begin{table*}[t]
	\centering
	\begin{tabular}{c||c|c|c|c|c|c|c|c}
		\hline\hline 
		Disturbance & $\Delta_{\psi_{\bfx_N},1}$ & $D_1$ & $\Delta_{\psi_{\bfx_N},2}$ & $D_2$ & $\Delta_{\psi_{\bfx_N},3}$ & $D_3$ & $\Delta_{\psi_{\bfx_N},4}$ & $D_4$\\\hline\hline
		Gaussian~(\ref{subsec:gauss}) & $3.83\mathrm{E}-4$ & $2.6\mathrm{E}-3$ & $1.69$ & $10.63$ & $3.36\mathrm{E}-6$ & $2.11\mathrm{E}-5$ & $4.29$ & $33.38$ \\\hline
		Laplace~(\ref{subsec:laplace}) & $0.037$ & $0.346$ & $0.077$ & $0.6838$ & $0.096$ & $0.270$ & $4.46$ & $33.16$ \\\hline
		Gaussian Mixture~(\ref{subsec:multiGauss})& 0.002& 0.016 & 0.001 & 0.011 & $2.94 E-4$& 0.031 & 4.94 & 31.87 \\\hline
	\end{tabular}
	\caption{The largest deviation between the actual and desired pdfs \eqref{eq:joint_matching_indiv}, compared to the $L_1$ distance \eqref{eq:marginal_matching} for each scenario.
	The controller from \eqref{opt:resulting_optimization} yields an $L_1$ distance that upper bounds the largest deviation in each case.}
	\label{tab:pdf_deviation}
\end{table*}

\begin{table*}[t]
	\centering
	\begin{tabular}{c||c|c||c|c|c|c}
		\hline\hline
		Disturbance & $J(K,v)$ & $J_{\textrm{MC}}(K,v)$ & $\Delta_X$ & $\Delta_{X,\textrm{MC}}$ & $\Delta_U$ & $\Delta_{U,\textrm{MC}}$ \\\hline\hline
		Gaussian~(\ref{subsec:gauss}) & $53.17$ & $49.36$ & $0.098$ & $0.052$ & $0.0975$ & $7\mathrm{E}-4$ \\\hline
		Laplace~(\ref{subsec:laplace}) & $41.84$ & $41.81$ & $0.0983$ & $0.0572$ & $0.0977$ & $0.0103$ \\\hline
		Gaussian Mixture~(\ref{subsec:multiGauss}) & $49.74$ & $46.18$ & $0.098$ & $0.009$ & $0.098$ & $0.015$ \\\hline
	\end{tabular}
	\caption{Cost and risk allocation (for the state and the input) when solving \eqref{opt:resulting_optimization}, and averaged values from $10^4$ Monte-Carlo (MC) samples for validation.
	The MC average cost is consistent with the computed cost, and the MC state and input constraint violations are lower than the computed violations.}
	\label{tab:cost_violation}
\end{table*}

\begin{figure*}[htb!]
	\centering
	\begin{subfigure}{\textwidth}
		\centering
		\includegraphics[scale=0.97]{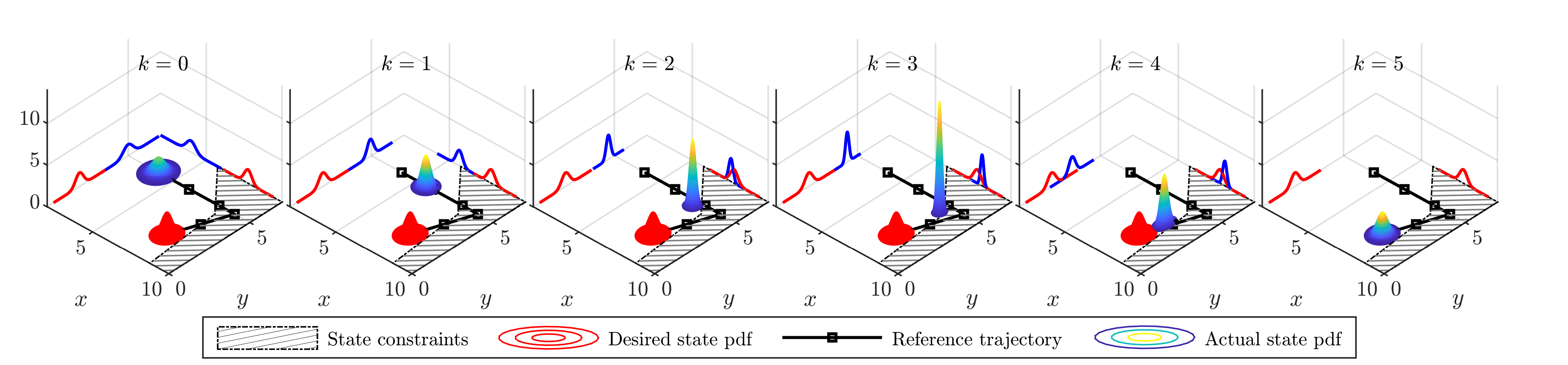}
		\caption{State evolution ($x$ and $y$) over 5 timesteps, subject to state chance constraints and terminal density constraints. The system is steered from the initial density to the final, desired density without collision, even though the reference trajectory violates the constraints.}
		\label{fig:gaussian2Dstate}
	\end{subfigure}
	\begin{subfigure}{\textwidth}
		\centering
		\includegraphics[scale=0.97]{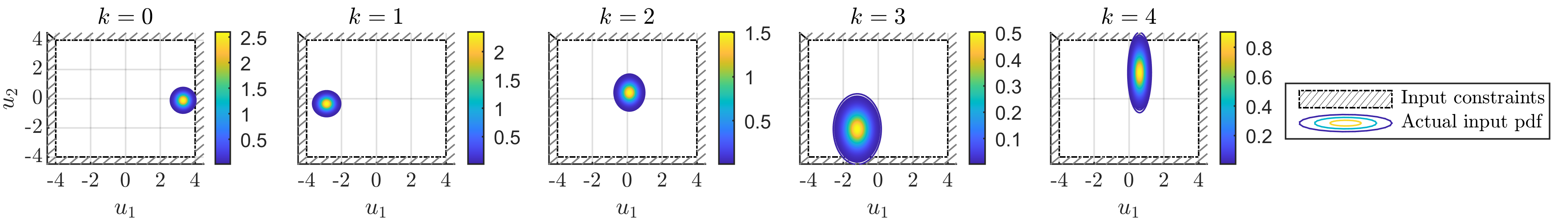}
		\caption{Inputs $u_1$ and $u_2$ satisfy input chance constraints with violation less than $\Delta_U$.}
		\label{fig:gaussian2Dinput}
	\end{subfigure}
	\caption{Distribution steering from one Gaussian distribution to another Gaussian distribution.}
\end{figure*}

\begin{figure*}[t!]
	\centering
	\begin{subfigure}{\textwidth}
		\centering
		\includegraphics[scale=0.97]{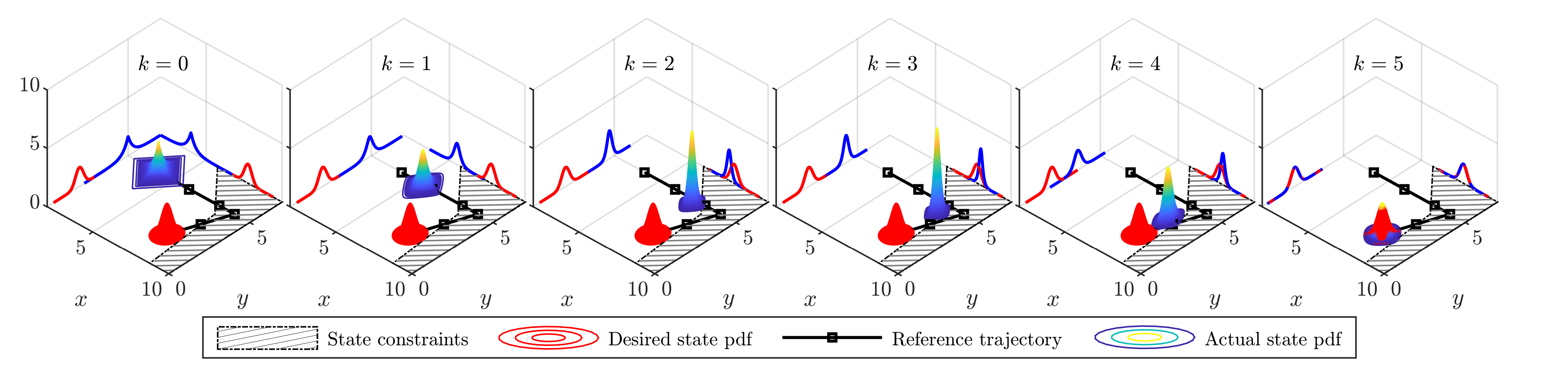}
		\caption{State evolution ($x$ and $y$) over 5 time steps, subject to state chance and terminal density constraints. The Laplace distribution is non-smooth at its peak, and is heavy-tailed. Our approach drives the system from a Laplace distribution to a Gaussian distribution, while maintaining state constraint violation below $\Delta_X$.}
		\label{fig:laplace2Dstate}
	\end{subfigure}
	\begin{subfigure}{\textwidth}
		\centering
		\includegraphics[scale=0.97]{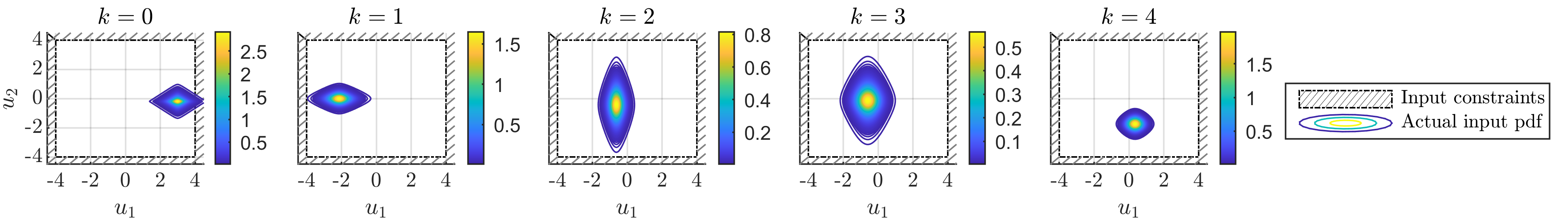}
		\caption{Inputs $u_1$ and $u_2$ satisfy input chance constraints with violation less than $\Delta_U$.}
		\label{fig:laplace2Dinput}
	\end{subfigure}
	\caption{Distribution steering from a  Laplace distribution to a Gaussian distribution.}
\end{figure*}

\begin{figure*}[t!]
	\centering
	\begin{subfigure}{\textwidth}
		\centering
		\includegraphics[scale=0.97]{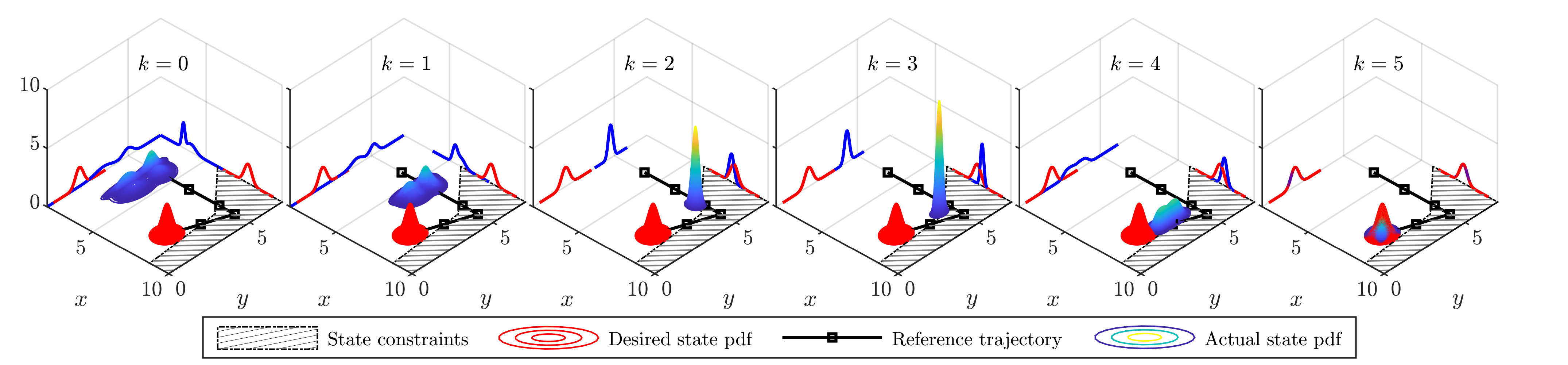}
		\caption{State evolution ($x$ and $y$) over timesteps, subject to the state chance and terminal density constraints. The multi-modal Gaussian is transformed into a Gaussian with a single mode with minimal state constraint violation.}
		\label{fig:multiGaussian2Dstate}
	\end{subfigure}
	\begin{subfigure}{\textwidth}
		\centering
		\includegraphics[scale=0.97]{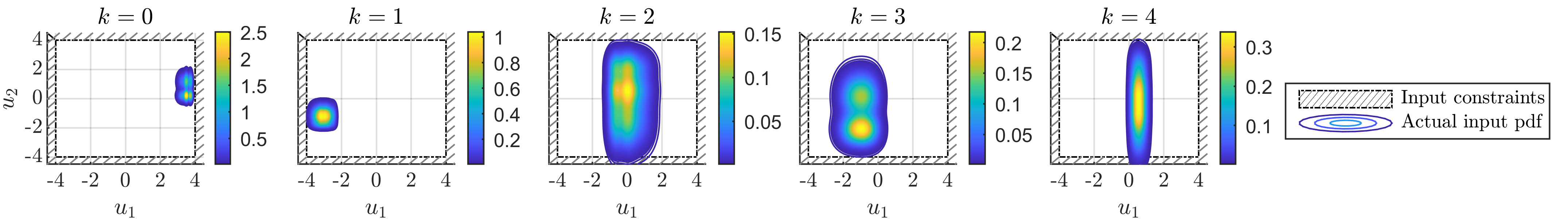}
		\caption{Inputs $u_1$ and $u_2$ satisfy input chance constraints with violation less than $\Delta_U$.}
		\label{fig:multiGaussian2Dinput}
	\end{subfigure}
	\caption{Distribution steering from a  Gaussian mixture to a Gaussian distribution.}
\end{figure*}

\bibliographystyle{IEEEtran}
\bibliography{IEEEabrv,shortIEEE,refs}

\begin{thebibliography}{10}
\providecommand{\url}[1]{#1}
\csname url@samestyle\endcsname
\providecommand{\newblock}{\relax}
\providecommand{\bibinfo}[2]{#2}
\providecommand{\BIBentrySTDinterwordspacing}{\spaceskip=0pt\relax}
\providecommand{\BIBentryALTinterwordstretchfactor}{4}
\providecommand{\BIBentryALTinterwordspacing}{\spaceskip=\fontdimen2\font plus
\BIBentryALTinterwordstretchfactor\fontdimen3\font minus
  \fontdimen4\font\relax}
\providecommand{\BIBforeignlanguage}[2]{{%
\expandafter\ifx\csname l@#1\endcsname\relax
\typeout{** WARNING: IEEEtran.bst: No hyphenation pattern has been}%
\typeout{** loaded for the language `#1'. Using the pattern for}%
\typeout{** the default language instead.}%
\else
\language=\csname l@#1\endcsname
\fi
#2}}
\providecommand{\BIBdecl}{\relax}
\BIBdecl

\bibitem{Halder2}
K.~F. Caluya and A.~Halder, ``Reflected {Schrödinger} bridge: Density control
  with path constraints,'' in \emph{American Control Conference}, New Orleans,
  LA, 2021, pp. 1137--1142.

\bibitem{balci2020covariance}
I.~M. Balci and E.~Bakolas, ``Covariance steering of discrete-time stochastic
  linear systems based on wasserstein distance terminal cost,'' \emph{IEEE
  Control Systems Letters}, vol.~5, no.~6, pp. 2000--2005, 2020.

\bibitem{Chen1}
Y.~Chen, T.~T. Georgiou, and M.~Pavon, ``Optimal steering of a linear
  stochastic system to a final probability distribution -- {Part} {I},''
  \emph{IEEE Trans. Automatic Control}, vol.~61, no.~5, pp. 1158--1169, 2016.

\bibitem{Williams2}
G.~Williams, P.~Drews, B.~Goldfain, J.~M. Rehg, and E.~A. Theodorou,
  ``Information-theoretic model predictive control: Theory and applications to
  autonomous driving,'' \emph{IEEE Transactions on Robotics}, vol.~34, no.~6,
  pp. 1603--1622, 2018.

\bibitem{Williams1}
\BIBentryALTinterwordspacing
G.~Williams, A.~Aldrich, and E.~A. Theodorou, ``Model predictive path integral
  control: From theory to parallel computation,'' \emph{Journal of Guidance,
  Control, and Dynamics}, vol.~40, no.~2, pp. 344--357, 2017. [Online].
  Available: \url{https://doi.org/10.2514/1.G001921}
\BIBentrySTDinterwordspacing

\bibitem{Max1}
M.~Goldshtein and P.~Tsiotras, ``Finite-horizon covariance control of linear
  time-varying systems,'' in \emph{56th IEEE Conference on Decision and
  Control}, Melbourne, Australia, Dec 12--15 2017, pp. 3606--3611.

\bibitem{pp_kazu}
K.~Okamoto and P.~Tsiotras, ``Optimal stochastic vehicle path planning using
  covariance steering,'' \emph{IEEE Robotics and Automation Letters}, vol.~4,
  no.~3, pp. 2276--2281, 2019.

\bibitem{bakolas2018finite}
E.~Bakolas, ``Finite-horizon covariance control for discrete-time stochastic
  linear systems subject to input constraints,'' \emph{Automatica}, vol.~91,
  pp. 61--68, 2018.

\bibitem{JP1}
J.~Pilipovsky and P.~Tsiotras, ``Chance-constrained optimal covariance steering
  with iterative risk allocation,'' in \emph{American Control Conference}, New
  Orleans, LA, 2021, pp. 2011--2016.

\bibitem{JoshJack}
J.~Ridderhof, J.~{Pilipovsky}, and P.~{Tsiotras}, ``Chance-constrained
  covariance control for low-thrust minimum-fuel trajectory optimization,'' in
  \emph{AAS/AIAA Astrodynamics Specialist Conference}, Lake Tahoe, CA, Aug
  9--13 2020.

\bibitem{PP_blackmore}
L.~Blackmore, H.~X. Li, and B.~C. Williams, ``A probabilistic approach to
  optimal robust path planning with obstacles,'' in \emph{American Control
  Conference}, Minneapolis, MN, June 14--16, 2006, pp. 1--7.

\bibitem{Boole}
A.~Prékopa, ``{Boole-Bonferroni} inequalities and linear programming,''
  \emph{Operations Research}, vol.~36, no.~1, pp. 145--162, 1988.

\bibitem{sivaramakrishnan2020convexified}
\BIBentryALTinterwordspacing
V.~Sivaramakrishnan, A.~P. Vinod, and M.~Oishi, ``Convexified open-loop
  stochastic optimal control for linear non-gaussian systems,'' \emph{IEEE
  Transactions on Automatic Control, (Submitted)}. [Online]. Available:
  \url{https://arxiv.org/abs/2010.02101}
\BIBentrySTDinterwordspacing

\bibitem{ECF}
V.~Sivaramakrishnan and M.~Oishi, ``Fast, convexified stochastic optimal
  open-loop control for linear systems using empirical characteristic
  functions,'' \emph{IEEE Control Systems Letters}, vol.~4, no.~4, pp.
  1048--1053, 2020.

\bibitem{billingsley_probability_2012}
P.~Billingsley, \emph{Probability and Measure}.\hskip 1em plus 0.5em minus
  0.4em\relax Wiley, 2008.

\bibitem{okamoto2018optimal}
K.~Okamoto, M.~Goldshtein, and P.~Tsiotras, ``Optimal covariance control for
  stochastic systems under chance constraints,'' \emph{IEEE Control Systems
  Letters}, vol.~2, no.~2, pp. 266--271, 2018.

\bibitem{lukacs_characteristic_1970}
E.~Lukacs, \emph{Characteristic Functions}, 2nd~ed.\hskip 1em plus 0.5em minus
  0.4em\relax London: Griffin, 1970.

\bibitem{ushakov_selected_1999}
N.~G. Ushakov, \emph{\BIBforeignlanguage{eng}{Selected Topics in Characteristic
  Functions}}, ser. Modern probability and statistics.\hskip 1em plus 0.5em
  minus 0.4em\relax Utrecht: VSP, 1999, no.~4.

\bibitem{cramer_mathematical_1999}
H.~Cramér, \emph{Mathematical Methods of Statistics}, ser. Princeton Landmarks
  in Mathematics and Physics.\hskip 1em plus 0.5em minus 0.4em\relax Princeton:
  Princeton University Press, 1999.

\bibitem{gil-pelaez_note_1951}
J.~Gil-Pelaez, ``\BIBforeignlanguage{en}{Note on the inversion theorem},''
  \emph{\BIBforeignlanguage{en}{Biometrika}}, vol.~38, no. 3-4, pp. 481--482,
  1951.

\bibitem{witkovsky_numerical_2016}
V.~Witkovsky, ``Numerical inversion of a characteristic function: {An}
  alternative tool to form the probability distribution of output quantity in
  linear measurement models,'' \emph{ACTA IMEKO}, vol.~5, no.~3, pp. 32--44,
  2016.

\bibitem{ono2008iterative}
M.~Ono and B.~Williams, ``Iterative risk allocation: A new approach to robust
  model predictive control with a joint chance constraint,'' in \emph{47th IEEE
  Conference on Decision and Control}, 2008, pp. 3427--3432.

\bibitem{farina_stochastic_2016}
M.~Farina, L.~Giulioni, and R.~Scattolini, ``\BIBforeignlanguage{en}{Stochastic
  linear model predictive control with chance constraints – a review},''
  \emph{\BIBforeignlanguage{en}{J. of Process Control}}, vol.~44, pp. 53--67,
  Aug. 2016.

\bibitem{mesbah2016stochastic}
A.~Mesbah, ``Stochastic model predictive control: An overview and perspectives
  for future research,'' \emph{{IEEE} Control Syst. Mag.}, vol.~36, no.~6, pp.
  30--44, 2016.

\end{thebibliography}

\appendix

\section*{A.~Proof of Theorem~\ref{thm:joint_matching}}

\setcounter{equation}{0}
\renewcommand{\theequation}{A.\arabic{equation}}

\begin{proof}
Since 
\begin{align}
	&  \Big| \big| \int_{\R^n}\exp(\i t^\intercal z)\upvarphi_{\bfx_N}(t) \d t \big| - \big|\int_{\R^n}\exp(\i t^\intercal z) \upvarphi_{\bfx_f}(t) \d t \big| \Big| \nonumber\\
	&\leq \int_{\R^n} \left |\exp(\i t^\intercal z)\upvarphi_{\bfx_N}(t) - \exp(\i t^\intercal z)\upvarphi_{\bfx_f}(t)\right|  \d t,
\end{align}
by
multiplying both sides by $({1}/{2\pi})^n$ and using \eqref{eq:inv_FT}, yields
\begin{equation}
    |\psi_{\bfx_N}(z) - \psi_{\bfx_f}(z)| \leq \left(\frac{1}{2\pi}\right)^n  \|\upvarphi_{\bfx_N} - \upvarphi_{\bfx}\|_{1}.
\end{equation}
Lastly, since this holds \textit{for all} $z\in\R^{n}$, we get \eqref{eq:joint_matching_1}.
\end{proof}

\section*{B.~Proof of Theorem~\ref{Theorem4}}
\setcounter{equation}{0}
\renewcommand{\theequation}{B.\arabic{equation}}

\begin{proof}
By definition of the $L_1$ distance, for each $i\in\mathbb{N}_{[1,N]}$,
\begin{align}
    \frac{1}{2 \pi}\int_{\R}|\upvarphi_{e_{i,n}^\intercal \bfx_{N}} (t_i) - \upvarphi_{\bfx_{f,i}} (t_i) | \ \d t_i \leq \epsilon_i. \label{eq:indiv_bounds}
\end{align}
Let $a_i = \upvarphi_{e_{i,n}^\intercal \bfx_{N}}(t_i)$ and $b_i = \upvarphi_{\bfx_{f,i}}(t_i)$.
Then, since $|\upvarphi_{e_{i,n}^\intercal \bfx_{N}}(t_i)|,\ |\upvarphi_{\bfx_{f,i}}(t_i)|\leq 1$, it follows that
$|\upvarphi_{\bfx_{N}}(t) - \upvarphi_{\bfx_f}(t)| = |\prod_i \upvarphi_{e_{i,n}^\intercal E_N \bfX}(t_i) - \prod_i\upvarphi_{\bfx_{f,i}}(t_i)| \leq \sum_i |\upvarphi_{e_{i,n}^\intercal \bfx_{N}}(t_i) - \upvarphi_{\bfx_{f,i}}(t_i)|$,
where we have used the fact that 
for $a_i,\ b_i\in\mathbb{C},\ i \in\mathbb{N}_{[1,n]}$ where $|a_i|,|b_i|\leq 1$, $|\prod_i a_i - \prod_i b_i| \leq \sum_i |a_i - b_i|$. 
The result now follows immediately from the definition of $D(K,v)$.
\end{proof}

\end{document}